\documentclass[oneside,english]{amsart}
\usepackage[T1]{fontenc}
\usepackage[latin9]{inputenc}
\usepackage{verbatim, graphicx, xcolor}
\usepackage{refstyle}
\usepackage{units}
\usepackage{mathrsfs}
\usepackage{amstext}
\usepackage{amsthm}
\usepackage{amssymb}
\usepackage{esint}

\makeatletter
\numberwithin{equation}{section}
\numberwithin{figure}{section}
\theoremstyle{plain}
\newtheorem{thm}{\protect\theoremname}[section]
\theoremstyle{definition}
\newtheorem{defn}[thm]{\protect\definitionname}
\theoremstyle{plain}
\newtheorem{lem}[thm]{\protect\lemmaname}
\theoremstyle{remark}
\newtheorem{rem}[thm]{\protect\remarkname}
\theoremstyle{remark}
\newtheorem*{acknowledgement*}{\protect\acknowledgementname}
\theoremstyle{plain}
\newtheorem{prop}[thm]{\protect\propositionname}
\theoremstyle{definition}
\newtheorem{example}[thm]{Example}

\theoremstyle{plain}
\newtheorem{conj}[thm]{\protect\conjecturename}

\def\rank{\mathrm{rank}}

\def\sl{\mathrm{SL}}

\usepackage{ifpdf} 
\ifpdf 

 \IfFileExists{lmodern.sty}{\usepackage{lmodern}}{}

\fi 

\makeatother

\usepackage{babel}
\providecommand{\definitionname}{Definition}
\providecommand{\lemmaname}{Lemma}
\providecommand{\propositionname}{Proposition}
\providecommand{\remarkname}{Remark}
\providecommand{\theoremname}{Theorem}
\providecommand{\conjecturename}{Conjecture}

\date{\today}

\begin{document}
\global\long\def\bbc{\mathbb{C}}
\global\long\def\bbr{\mathbb{R}}
\global\long\def\bbq{\mathbb{Q}}
\global\long\def\bbz{\mathbb{Z}}
\global\long\def\bbn{\mathbb{N}}
\global\long\def\norm#1{\left\Vert #1\right\Vert}
\global\long\def\frakg{\mathfrak{g}}
\global\long\def\fraka{\mathfrak{a}}
\global\long\def\fraku{\mathfrak{u}}
\global\long\def\frakt{\mathfrak{t}}
\global\long\def\fraks{\mathfrak{s}}
\global\long\def\frakn{\mathfrak{n}}

\title[Existence of Non-Obvious Divergent Trajectories]{Existence of Non-Obvious Divergent Trajectories in homogeneous spaces}
\author{Nattalie Tamam}
\address{Mathematics Department, UC San Diego, 9500 Gilman Dr, La Jolla, CA 92093}
\email{natamam@ucsd.edu}

\begin{abstract}
We prove a modified version for a conjecture of Weiss from 2004. Let $G$ be a semisimple real algebraic group defined over $\bbq$, $\Gamma$ be an arithmetic subgroup of $G$. A trajectory in $G/\Gamma$ is divergent
if eventually it leaves every compact subset, and is obvious divergent if there is a finite collection of algebraic data which cause the divergence. 
Let $A$ be a diagonalizable subgroup of $G$ of positive dimension. We show that if the projection of $A$ to any $\bbq$-factor of $G$ is of small enough dimension (relatively to the $\bbq$-rank of the $\bbq$-factor), then there are non-obvious divergent trajectories for the action of $A$ on $G/\Gamma$. 
\end{abstract}

\maketitle

\section{Introduction}

Let $G$ be a semisimple real algebraic group defined over $\bbq$,  $\Gamma$ be an arithmetic subgroup of $G$, and $A\subset G$ be a subgroup. The action of $A$ on $G/\Gamma$ induces a flow on $G/\Gamma$. The behavior of such flows is extensively studied and related to classical problems in number theory (see \cite{key-15}).

For example, it was proved by Dani \cite{D} that divergent trajectories are related to singular systems of linear forms which are studied in the theory of Diophantine approximation.
A trajectory $Ax$ in $G/\Gamma$ is called \textbf{divergent} if the map $a\mapsto ax$, $a\in A$, is proper. 

In some cases one can find a simple algebraic reason for the divergence. 
Let $g\in G$ and let $A\subset G$ be a semigroup. A trajectory $Ag\Gamma$ is called an \textbf{obvious divergent trajectory} if for any unbounded sequence $\left\{ a_{k}\right\} \subset A$ there is a sub-sequence $\left\{ a_{k}^{\prime}\right\} \subset\left\{ a_{k}\right\} $, a $\mathbb{Q}$-representation $\varrho:G\rightarrow\mbox{GL}\left(V\right)$, and a nonzero $v\in V\left(\mathbb{Q}\right)$ such that  \[
\varrho\left(a_{k}^{\prime}g\right)v\underset{k\rightarrow+\infty}{\longrightarrow}0.
\]
A proof that an obvious divergent trajectory is indeed divergent can be found in \cite{w}. Obvious divergent trajectories are related to systems of linear forms with coefficients that lie in a rational hyperspace.  

In \cite{m} it was shown by Margulis that a unipotent subgroup has no divergent trajectories on $G/\Gamma$. Moreover, his argument shows that any quasi-unipotent subsemigroup has no divergent trajectories. Thus, it is natural to study the existence of divergent trajectories and non-obvious divergent trajectories under the action of diagonalizable subgroups of $G$. Since any diagonalizable subgroup of $G$ is a direct product of a compact set and an $\bbr$-diagonalizable subgroup, we focus on the latter. 

Let $T$ be a maximal $\bbr$-split torus in $G$. It was conjectured by Weiss in \cite{w} that the existence of divergent trajectories and non-obvious divergent trajectories for the action of a subgroup of $T$ on $G/\Gamma$ can be deduced from the relation between the $\bbq$-rank of $G$ and the dimension of the subgroup.  

\begin{conj} \label{conj:weiss}\cite[Conjecture 4.11(C)]{w}
	Let $A$ be a subgroup of $T$. If $0<\dim A<\mathrm{rank}_\bbq G$ then there exist non-obvious divergent trajectories for the action of $A$ on $G/\Gamma$.
\end{conj}

A connected algebraic semi-simple $\bbq$-group is called \textbf{almost $\bbq$-simple} if it has no closed connected normal $\bbq$-subgroup of strictly positive dimension. 
It follows from Wiess work that the above conjecture does not hold when $G$ is not an almost $\bbq$-simple group, as we would see next.

Let $G_1, G_2$ be semisimple real algebraic groups defined over $\bbq$ which satisfy \[
\mathrm{rank}_\bbq(G_1)=1,\quad\mathrm{and}\quad \mathrm{rank}_\bbr(G_1)=\mathrm{rank}_\bbq(G_2)=\mathrm{rank}_\bbr(G_2)=2\]
(such groups exist, e.g. Table VI in \cite{he}). 
Let $\Gamma_{1},\Gamma_{2}$ be arithmetic subgroups of $G_1,G_2$, receptively, and $G=G_1\times G_2$, $\Gamma=\Gamma_{1}\times\Gamma_{2}$. 
Let $T_1$ be an $\bbr$-maximal split torus of $G_1$, and $A=T_1\times\{e\}$. Then \[
\dim A=2<3=\mathrm{rank}_\bbq G.\] 
However, since \[
\dim A_1=2>1=\mathrm{rank}_\bbq G_1,\] 
according to \cite[Corollary 2]{w Q-rank} there are no divergent trajectories for the action of $A_1$ on $G_1/\Gamma_1$. Thus, there are no divergent trajectories for the action of $A$ on $G/\Gamma$ (see Lemma \ref{lem:factor equiv}).    
 
In this work we show that Conjecture \ref{conj:weiss} holds when $G$ is almost $\bbq$-simple. We also find a sufficient condition for the existence of non-obvious divergent trajectories in the general, which depends on the projection of $A$ onto the $\bbq$-factors of $G$. 

A real algebraic $\bbq$-group $H$ is said to be the \textbf{$\bbq$-almost direct product} of its real algebraic $\bbq$-subgroups $H_{1},\dots,H_{\ell}$ if the map
\begin{align}
	H_{1}\times\cdots\times H_{\ell} & \rightarrow H\label{eq: almost direct product}\\
	\left(h_{1},\dots,h_{\ell}\right) & \mapsto h_{1}\cdots h_{\ell}\nonumber
\end{align} 
is a surjective homomorphism with finite kernel; in particular, this means that the $H_{i}$ commute with each other, and each $H_{i}$ is normal in $H$. The subgroups $H_{1},\dots,H_{\ell}$ are called \textbf{$\bbq$-factors of $H$}. If $T'$ is a maximal $\bbr$-split torus in $H$, then for any $i$, the identity connected component $T'_i:=(T'\cap H_i)^\circ$ is a maximal $\bbr$-split torus of $H_i$. Moreover, $T'$ is the almost direct product of the $T'_i$'s (see \cite[Proposition 22.9]{Bor}). 

We start by showing a sufficient and necessary condition for the existence of divergent trajectories. 

\begin{thm} \label{thm:2}
	Let $A$ be a non-trivial subgroup of $T$. Then, there exist divergent trajectories for the action of $A$ on $G/\Gamma$ if and only if for any $\bbq$-factor $G_1$ of $G$
	\begin{equation} \label{eq:main}
	\dim(A\cap G_1)\leq\mathrm{rank}_\bbq G_1.
	\end{equation}
\end{thm}

The necessity of (\ref{eq:main}) for the existence of divergent trajectories was shown by Weiss, when $G$ is almost $\bbq$-simple (see \cite[Corollary 2]{w Q-rank}).

Taking into consideration the constrains of Theorem \ref{thm:2} and the effect of each $\bbq$-factor on the existence of non-obvious divergent trajectories, our main result is the following modified version of Conjecture \ref{conj:weiss}.

\begin{thm}
	\label{thm:main} Let $A$ be a subgroup of $T$ such that any $\bbq$-factor $G_1$ of $G$ satisfies (\ref{eq:main}) and there exists a $\bbq$-factor $G_2$ of $G$ such that \[
	0<\dim(A\cap G_2)<\mathrm{rank}_\bbq G_1,\]
	then there are non-obvious divergent trajectories for the action of $A$ on $G/\Gamma$.
\end{thm}


When $A$ is one-dimensional, Theorem \ref{thm:main} was proved in \cite[Prop. 4.5]{D} and in \cite[Prop. 3.5]{w} for the case $A$ is not quasi-unipotent (not necessarily a subgroup of $T$). 
For the case $G=\mathrm{SL}_n(\bbr)$, $\Gamma=\mathrm{SL}_n(\bbz)$ and $A$ is a subgroup of the diagonal group in which there is a fixed element on the diagonal that is unbounded for any unbounded sequence of $A$, Theorem \ref{thm:main} was proved in \cite[Corollary 4.14]{w}.   

A main step in the results of Dani and Wiess is showing that non-obvious divergent trajectories can be constructed assuming one can find `enough' obvious reasons for the divergence of the orbit $A\Gamma$. 
The construction is based on ideas of Khintchine \cite{Khintchine}, which were later further developed by Cassels \cite{Cassels}, Dani \cite{D}, and Weiss \cite{w}. We use the same scheme, here in the form of \cite[Theorem 4.13]{w}.  

\subsection{Overview of the paper}
We start by defining almost $\bbq$-anisotropic subgroups and almost $\bbq$-split subtori and show that if a subtorus is of a small enough dimension, then it is conjugate to a maximal almost $\bbq$-split subgroup (Theorem \ref{thm:maximal almost split}) in \S\ref{sec:maximal almost}.  
In \S \ref{sec:Highest Weight Representations} we show Theorem \ref{thm:2} using induction on the number of $\bbq$-factors of $G$.  
The base of the induction contains two implications. One is a theorem of Weiss and the other follows from Theorem \ref{thm:maximal almost split}.

The case $G=\sl_4(\bbr)$, $\Gamma=\sl_4(\bbz)$, and $A=\left\{\exp(\mathrm{diag}(s,-s,t,-t))\::\:s,t\in\bbr\right\}$ is a special case of Theorem \ref{thm:main} which is not considered in \cite[Corollary 4.14]{w}. In \S \ref{sec: example} we show the existence of non-obvious divergent trajectories for it, as a simplified version of the proof.

The proof of Theorem \ref{thm:main} is in \S \ref{sec:Rational Characters and Divergent Trajectories}. It builds on the Khintchine-Cassels-Dani-Weiss scheme, here in the form of Theorem \ref{thm: weiss for epsilon case}, to produce non-obvious divergent trajectories.  
It follows from Theorem \ref{thm:maximal almost split} that when the dimension of $A$ is small enough, one can assume that there is a non-trivial $\bbq$-character which is trivial on $A$, but no non-trivial subgroup of $A$ such that all $\bbq$-characters are trivial on it.
We then use this character to construct 'reasons' for divergence and verify the conditions for applying the scheme.
The main novelty of the proof is using a basic result- that every fundamental weight is conjugated under the Weyl group to a dominant weight, and the lower rank of the subgroup to find 'guaranteed' unbounded $\bbq$-abstract weights and construct these 'reasons' for divergence. 
\newline

\noindent\emph{Acknowledgements:} I would like to thank Barak Weiss for introducing me to this beautiful topic. I am grateful to the anonymous referee for constructive comments and suggestions. The first author was partially supported by the Eric and Wendy Schmidt Fund for Strategic Innovation.

\section{\label{sec:maximal almost}A Maximal Almost $\bbq$-Split Torus}

Let $S$ be a maximal $\bbq$-split torus. By replacing $T$ with a conjugate of it we may assume
$S\subset T$. A \textbf{character} defined on $T$ is a linear functional on it, and we say it is a \textbf{$\bbq$-character} if its restriction to $S$ is not trivial. 

\begin{defn}\label{def:almost}
	Let $A$ be a subgroup of $T$ (not necessarily defined over $\bbq$).
	We say that $A$ is \textbf{almost $\bbq$-anisotropic} if all $\bbq$-character defined on $T$ are trivial on $A$, and that $A$ is \textbf{almost $\bbq$-split} if it does not contain non-trivial almost $\bbq$-anisotropi subgroups. If $A$ is almost $\bbq$-anisotropic (respectively almost $\bbq$-split) and defined over $\bbq$, it is called \textbf{$\bbq$-anisotropic}	(respectively $\bbq$\textbf{-split}). 
\end{defn}
The goal of this chapter is to show that if $G$ is almost $\bbq$-simple, then any subgroup of $T$ is conjugated to a subgroup with an almost $\bbq$-split subtorus of maximal dimension. This will be used later on to prove Theorem \ref{thm:2} and Theorem \ref{thm:main}. 
\begin{thm}\label{thm:maximal almost split}
	Assume $G$ is almost $\bbq$-simple. Let $A$ be a subgroup of $T$ such that $\dim\left(A\right)\leq\text{rank}_{\bbq}\left(G\right)$.
	Then there exists $n\in N_{G}\left(T\right)$ such that $\text{Ad}\left(n\right)A$
	is almost $\bbq$-split.
\end{thm}

We use standard notions and results of linear algebraic groups (see
\cite{Bor}).

Denote by $\frakg$, $\fraks$, $\frakt$ the Lie algebras of
$G$, $S$, $T$, respectively.
The dimension of $\frakt$ (resp. $\fraks$) is denoted by $\rank_{\bbr} G$ (resp. $\rank_\bbq G$). We denote by $\frakt^*$ (resp. $\fraks^*$) the group of $\bbr$-characters of $\frakt$, (respectively $\bbq$-characters of $\fraks$). $\fraks^*$ is exactly the restrictions to $\fraks$ of elements of $\frakt^*$. Characters are written additively.
Denote by $\Phi_{\bbq}$ and $\Phi_{\bbr}$ the set of $\bbq$-roots
in $\fraks^{*}$ and $\bbr$-roots in $\frakt^{*}$, respectively.

\begin{rem}
	Since in this section we assume that Assume $G$ is almost $\bbq$-simple, $\Phi_{\bbr}$ and $\Phi_\bbq$ are irreducible (i.e, can not be written as $\Phi=\Phi_{1}\oplus\Phi_{2}$, where $\Phi_{1}$, $\Phi_{2}$ are non-trivial root systems).
\end{rem}

Let $\kappa$ be the Killing form on $\frakg$. For $\chi\in\mathfrak{t}^{*}$
let $t_{\chi}\in\mathfrak{t}$ be determined by 
\begin{equation}
\chi\left(t\right)=\kappa\left(t_{\chi},t\right)\text{ for all }t\in\mathfrak{t}.\label{eq:t_chi def}
\end{equation}
For $\chi_{1},\chi_{2}\in\mathfrak{t}^{*}$, $\chi_{1}\neq0$, let 
\begin{equation}
\left\langle \chi_{1},\chi_{2}\right\rangle =2\frac{\kappa\left(t_{\chi_{1}},t_{\chi_{2}}\right)}{\kappa\left(t_{\chi_{1}},t_{\chi_{1}}\right)}.\label{eq:innr def}
\end{equation}

We may now view $\fraks^*$ as a subset of $\frakt^*$. Let $\frakt_0$ be the orthogonal compliment of $\fraks$ in $\frakt$ with respect to the above inner product. Then, we may define any $\chi\in\fraks^*$ on $\frakt$ by taking its restriction on $\frakt_0$ to be zero. 

Denote by $W\left(\Phi_{\bbr}\right)$ the Weyl group associated with
$\Phi_{\bbr}$, i.e. the group generated by the reflections $s_{\beta}$
, $\beta\in\Phi_{\bbr}$, defined by 
\begin{equation}
s_{\beta}\left(\chi\right)=\chi-\left\langle \chi,\beta\right\rangle\beta\label{eq: Weyl group def}
\end{equation}
for any characters $\chi$. For any $w\in W\left(\Phi_{\bbr}\right)$
there exists $n_{w}\in N_{G}\left(T\right)$ such that for all $\chi\in\mathfrak{t}^{*}$
\begin{equation}
\text{Ad}\left(n_{w}\right)t_{\chi}=t_{w\left(\chi\right)}\label{eq:weyl act}
\end{equation}
(see \cite[§VI.5]{k}). It follows from (\ref{eq:t_chi def}), (\ref{eq:weyl act}),
and the invariance of the Killing form under automorphisms of $\frakg$
that for any $\chi\in\frakt^{*}$, $t\in\frakt$, and $w\in W\left(\Phi_{\bbr}\right)$
\begin{equation}
w\left(\chi\right)\left(t\right)=\chi\left(\text{Ad}\left(n_{w}^{-1}\right)t\right)\label{eq:moving w}
\end{equation}

\begin{prop}
\label{prop:one step}For any non-zero $t\in\frakt$ and non-zero $\chi\in\frakt^{*}$
there exists $\beta\in\Phi_{\bbr}\cup\left\{ 0\right\} $ such that
$s_{\beta}\left(\chi\right)\left(t\right)\neq0$.
\end{prop}

\begin{proof}
If $\chi\left(t\right)\neq0$, then we may take $\beta=0$. 

Assume $\chi\left(t\right)=0$. 
The span of the $W\left(\Phi_{\bbr}\right)$-orbit of a non-zero character
$\chi$ is a nonzero $W\left(\Phi_{\bbr}\right)$-invariant subspace
of $\frakt^{*}$. According to Lemma B in \cite[§10.4]{H}, $W\left(\Phi_{\bbr}\right)$
acts irreducibly on $\frakt^{*}$. Therefore, the $W\left(\Phi_{\bbr}\right)$-orbit
of $\chi$ spans $\frakt^{*}$. Since $t$ is non-zero, there exists
$w\in W\left(\Phi_{\bbr}\right)$ such that $w\left(\chi\right)\left(t\right)\neq0$.

The element $w\in W\left(\Phi_{\bbr}\right)$ can be written as a product $w=s_{\lambda_{1}}\cdots s_{\lambda_{k}}$ with $\lambda_{1},\dots,\lambda_{k}\in\Delta_{\bbr}$.
Assume that $w$ is chosen so that $k$ is minimal. Then, by (\ref{eq: Weyl group def}) \[
\left\langle \chi,\lambda_{k}\right\rangle \neq0.\]
If $\lambda_k(t)\neq0$, then by (\ref{eq: Weyl group def}) we are done. But it may not be the case. We will see that 
\begin{equation}
\forall1\leq i<k,\qquad\left\langle \chi,\lambda_{i}\right\rangle =0,\label{eq:all i not k}
\end{equation}
and use it to choose $\beta$. 

Assume otherwise; for some $1\leq i<k$ we have 
\begin{equation}
\left\langle \chi,\lambda_{i}\right\rangle =c\neq0.\label{eq:c def}
\end{equation}
For $j=1,\dots,k$ denote $s_{j}=s_{\lambda_{j}}$. Recall that for any $s\in W\left(\Phi_{\bbr}\right)$ we denote by $n_{s}$ the element in $N_{G}\left(T\right)$ which represents $s$. If the adjoint action of $n_{s_{i}}^{-1}$ on $\text{Ad}\left(n_{s_{i-1}}^{-1}\cdots n_{s_{1}}^{-1}\right)t$ is trivial, then by (\ref{eq:moving w}) 
\begin{align*}
s_{1}\cdots s_{i-1}s_{i+1}\cdots s_{k}\left(\chi\right)\left(t\right) & =s_{i+1}\cdots s_{k}\left(\chi\right)\left(\text{Ad}\left(n_{s_{i-1}}^{-1}\cdots n_{s_{1}}^{-1}\right)t\right)\\
 & =s_{i+1}\cdots s_{k}\left(\chi\right)\left(\text{Ad}\left(n_{s_{i}}^{-1}n_{s_{i-1}}^{-1}\cdots n_{s_{1}}^{-1}\right)t\right)\\
 & =w\left(\chi\right)\left(t\right)\neq0,
\end{align*}
a contradiction to the minimality of $k$. 

Let $w^{\prime}=s_1\cdots s_{i-1}$ and $t^{\prime}=\text{Ad}\left(n_{w^{\prime}}^{-1}\right)t$.
Note that if $s$ is a reflection, it is its own inverse and so $n_{s}^{-1}$
is another represent of $s$. Since $\kappa\left(t^{\prime},\cdot\right)$
defines a character on $\frakt$, equations (\ref{eq:t_chi def}),
(\ref{eq:innr def}), (\ref{eq: Weyl group def}) and (\ref{eq:weyl act})
yield
\[
\text{Ad}\left(n_{s_{i}}^{-1}\right)t^{\prime}=t^{\prime}-\frac{2\kappa\left(t^{\prime},t_{\lambda_{i}}\right)}{\kappa\left(t_{\lambda_{i}},t_{\lambda_{i}}\right)}t_{\lambda_{i}}.
\]
Hence, the non-triviality of the adjoint action of $n_{s_{\lambda_{i}}}^{-1}$
on $t^{\prime}$ implies
\begin{equation}
2\lambda_{i}\left(t^{\prime}\right)=2\kappa\left(t^{\prime},t_{\lambda_{i}}\right)\neq0.\label{eq:k not 0}
\end{equation}

Equation (\ref{eq:moving w}) and the minimality of $k$ imply that $\chi\left(t^{\prime}\right)=w'\left(\chi\right)\left(t\right)=0$.
Therefore, (\ref{eq: Weyl group def}), (\ref{eq:moving w}), (\ref{eq:c def}),
and (\ref{eq:k not 0}) imply 
\begin{align*}
s_{1}\cdots s_{i}\left(\chi\right)\left(t\right) & =s_{i}\left(\chi\right)\left(t^{\prime}\right)\\
 & =\left(\chi+c\lambda_{i}\right)\left(t^{\prime}\right)\\
 & =\chi\left(t^{\prime}\right)+c\lambda_{i}\left(t^{\prime}\right)\neq0.
\end{align*}
A contradiction to the minimality of $k$, proving (\ref{eq:all i not k}). 

Let $\beta=s_{1}\cdots s_{k-1}\left(\lambda_{k}\right)$. Then, 
\begin{equation}
s_{\beta}=s_{1}\cdots s_{k-1}s_{k}s_{k-1}\cdots s_{1}\label{eq:s_beta}
\end{equation}
(see \cite[\S II.6]{k}). It follows from (\ref{eq: Weyl group def}), (\ref{eq:all i not k}), and
(\ref{eq:s_beta}) that 
\begin{align*}
s_{\beta}\left(\chi\right)\left(t\right) & =s_{1}\cdots s_{k-1}s_{k}\left(\chi\right)\left(t\right)=w\left(\chi\right)\left(t\right)\neq0.
\end{align*}
\end{proof}

\begin{lem}
\label{lem:max ani-spl}Let $A$ be a subgroup of $T$. Then $A$
contains a maximal almost $\bbq$-anisotropic subgroup $A_{\text{ani}}$
and a maximal almost $\bbq$-split subgroup $A_{\text{spl}}$ such
that $A=A_{\text{ani}}A_{\text{spl}}$ and $A_{\text{ani}}\cap A_{\text{spl}}$
is finite.
\end{lem}

\begin{proof}
According to \cite[Prop. 13.2.4]{s} there exists a unique maximal
$\bbq$-anisotropic subtorus $T_{\text{ani}}$ of $T$. Then, by definition,
the set $A_{\text{ani}}=A\cap T_{\text{ani}}$ is a maximal $\bbq$-anisotropic
subtorus of $A$. By \cite[Prop. 13.2.3]{s} there exists a subgroup
$A_{\text{spl}}$ of $A$ such that $A=A_{\text{ani}}A_{\text{spl}}$
and $A_{\text{ani}}\cap A_{\text{spl}}$ is finite. Note that $A_{\text{spl}}$ is not necessarily a subset of $S$ (and in that case, not defined over $\bbq$). Then, any subgroup
of $A$ properly containing $A_{\text{spl}}$ contains a non-trivial
almost $\bbq$-anisotropi subgroup. Hence, $A_{\text{spl}}$ is a
maximal almost $\bbq$-split subgroup of $A$.
\end{proof}

\begin{lem}\label{lem:exists Q-caracter}
	Let $V$ be an inner product space and $U,W$ be subspaces of $V$. If $\dim W<\dim U$, then there exists a non-trivial $\varphi\in V^*$ which is trivial on $W$ and on the orthogonal complement of $U$ in $V$. 
\end{lem}
\begin{proof}
	It follows from \[
	\dim W^*=\dim W<\dim U=\dim U^*\]
	that there exists a non-trivial $\varphi'\in U^*$ which is trivial on $W\cap U$. Denote by $p:V\rightarrow U$ the orthogonal projection of $V$ onto $U$. Define $\varphi\in V^*$ by \[
	\varphi(u)=\varphi'(p(u)).
	\]
	Then, $\varphi$ satisfies the claim. 
\end{proof}

\begin{proof}[Proof of Theorem \ref{thm:maximal almost split}]

According to Lemma \ref{lem:max ani-spl}, for any subgroup $T^{\prime}$ of $T$ we can denote by $T_{\text{ani}}^{\prime}$ its unique maximal almost $\bbq$-anisotropi subgroup and by $T_{\text{spl}}^{\prime}$ its unique maximal almost $\bbq$-split subgroup. If $\frakt^{\prime}$
is the Lie algebra of $T^{\prime}$, denote by $\frakt_{\text{ani}}^{\prime}$
and $\frakt_{\text{spl}}^{\prime}$ the Lie algebras of $T_{\text{ani}}^{\prime}$
and $T_{\text{spl}}^{\prime}$, respectively.

Let $\fraka$ be the Lie algebra of $A$. Without loss of generality,
assume that the dimension of $\fraka_{\text{ani}}$ is less than or
equal to the dimension of $\left(\text{Ad}\left(n\right)\fraka\right)_{\text{ani}}$
for any $n\in N_{G}\left(T\right)$. 

Assume by contradiction that $\fraka_{\text{ani}}$ has positive dimension. Then, $\dim\fraka_{\text{spl}}<\text{rank}_{\bbq}\left(G\right)$.
By Lemma \ref{lem:exists Q-caracter} there exists a $\bbq$-character $\chi$ defined on $\frakt$
which is trivial on $\fraka_{\text{spl}}$ (therefore trivial on $\fraka$).
Let $a$ be a non-zero element in $\fraka_{\text{ani}}$. By Proposition
\ref{prop:one step} there exists $\beta\in\Phi_{\bbr}$ such that
$s_{\beta}\left(\chi\right)\left(a\right)\neq0$. Note that since
$a\in\fraka_{\text{ani}}$, we have $\chi\left(a\right)=0$. Thus,
$s_{\beta}\left(\chi\right)\neq\chi$, which by (\ref{eq: Weyl group def})
implies 
\begin{equation}
\left\langle \chi,\beta\right\rangle \neq0.\label{eq:chi,beta not 0}
\end{equation}
We claim that 
\begin{equation}
\text{Ad}\left(n_{s_{\beta}}\right)\fraka_{\text{spl}},\left\{ \text{Ad}\left(n_{s_{\beta}}\right)a\right\} \subset\left(\text{Ad}\left(n_{s_{\beta}}\right)\fraka\right)_{\text{spl}}.\label{eq:dim cont}
\end{equation}

By (\ref{eq:moving w}) and the choice of $\beta$, $\chi\left(\text{Ad}\left(n_{s_{\beta}}\right)a\right)=s_{\beta}\left(\chi\right)\left(a\right)$
is non-zero. Hence 
\[
\text{Ad}\left(n_{s_{\beta}}\right)a\in\left(\text{Ad}\left(n_{s_{\beta}}\right)\fraka\right)_{\text{spl}}.
\]

Let $t\in\fraka_{\text{spl}}$. Then there exists a $\bbq$-character
$\lambda$ defined on $\frakt$ which is non-trivial on $t$. If $t$
is invariant under the adjiont action of $n_{s_{\beta}}$, then 
\[
\lambda\left(\text{Ad}\left(n_{s_{\beta}}\right)t\right)=\lambda\left(t\right)\neq0.
\]
If $t$ is not invariant under the adjiont action of $n_{s_{\beta}}$, then as in the proof of Proposition \ref{prop:one step} one can get that $\beta\left(t\right)\neq0$. Thus, by (\ref{eq:moving w}) and (\ref{eq:chi,beta not 0}) we arrive at 
\[
\chi\left(\text{Ad}\left(n_{s_{\beta}}\right)t\right)=s_{\beta}\left(\chi\right)\left(t\right)=\chi\left(t\right)-\left\langle \chi,\beta\right\rangle \beta\left(t\right)\neq0.
\]
We may conclude that for any $t\in\fraka_{\text{spl}}$, 
\[
\text{Ad}\left(n_{s_{\beta}}\right)t\in\left(\text{Ad}\left(n_{\beta}\right)\fraka\right)_{\text{spl}}.
\]

This proves (\ref{eq:dim cont}), a contradiction to the maximality
of the dimension of $\fraka_{\text{spl}}$. 
\end{proof}

\section{\label{sec:Highest Weight Representations}Highest Weight Representations}

The goal of this section is to prove Theorem \ref{thm:2}. We start with some notation and preliminary results about $\bbr$-representations, and specifically $\bbr$-representations with a highest weight. 

We preserve the notation of \S \ref{sec:maximal almost}. 

Let $\varrho:G\rightarrow\mbox{GL}\left(V\right)$ be a $\bbq$-representation of $G$. For $\lambda\in\fraks^{*}$ denote
\begin{equation*}\label{eq: defn Q-weight}
V_{\varrho,\lambda}=\left\{v\in V\::\:\forall s\in\fraks\quad\varrho\left(\exp\left(s\right)\right)v=e^{\lambda\left(s\right)}v\right\}.
\end{equation*}
If $V_{\varrho,\lambda}\neq\left\{ 0\right\}$, then $\lambda$ is called a \textbf{$\bbq$-weight for $\varrho$}. Denote by $\Phi_{\varrho}$
the set of $\bbq$-weights for $\varrho$. For any $\lambda\in\Phi_{\varrho}$,
$V_{\varrho,\lambda}$ is called \textbf{the $\bbq$-weight vector space for $\lambda$}, and members of $V_{\varrho,\lambda}$ are called \textbf{$\bbq$-weight vectors for $\lambda$}.

As in the previous section, $\Phi_{\bbq}$ is the set of $\bbq$-roots. For $\beta\in\Phi_{\bbq}\cup\left\{0\right\}$ denote by $\frakg_{\beta}$ the $\bbq$-root space for $\beta$, i.e., the $\bbq$-weight vector space for $\beta$ with respect to the adjoint representation. 

For any $\bbq$-representation $\varrho$, the $\bbq$-weights for $\varrho$ are related to the $\bbq$-roots by 
\begin{equation}
\varrho\left(\frakg_{\beta}\right)V_{\varrho,\lambda}\subset V_{\varrho,\beta+\lambda},\quad\beta\in\Phi_\bbq \cup\left\{0\right\},\;\lambda\in\Phi_{\varrho}.\label{eq: lambda acts on chi}
\end{equation}


For any $\chi\in\frakt^*$ and $\frakt'\subset \frakt$ denote by $\chi\mid_{\frakt'}$ the restriction of $\chi$ to $\frakt'$. 

Recall that $\frakt_0$ is the orthogonal compliment with respect to the inner product defined in (\ref{eq:innr def}). Then, by definition any $\bbq$-character defined on $\frakt$ is trivial on $\frakt_0$.

\begin{lem}\label{lem:chi tilde}
	Let $\varrho:G\rightarrow\mbox{GL}\left(V\right)$ be a $\bbq$-representation with a $\bbq$-weight $\chi$ such that $\dim\left(V_{\varrho,\chi}\right)=1$. Then, for any $t\in\frakt$, $v\in V_{\varrho,\chi}$\[
	\varrho\left(\exp\left(t\right)\right)v=e^{{\chi}\left(t\right)}v.\]
\end{lem}

Note that the set $V_{\varrho,\lambda}$ is defined with respect to $\fraks$ (and not $\frakt$). 

\begin{proof}
	Since $V_{\varrho,\chi}$ is one dimensional and $\frakt\subset\frakg_0$, Equation (\ref{eq: lambda acts on chi}) implies that the torus $T$ acts multiplicatively on $V_{\varrho,\chi}$. That is, there is a charcter $\lambda\in\frakt^*$ such that for any $t\in\frakt$, $v\in V_{\varrho,\chi}$\[
	\varrho\left(\exp\left(t\right)\right)v=e^{\lambda\left(t\right)}v.\]	
	By the assumption, $\lambda\mid_{\fraks}=\chi$. 
	Since $V_{\varrho,\chi}$ is a $\bbq$-weight space, it is defined over $\bbq$. Hence, $\lambda$ is a $\bbq$-character and and we may deduce $\lambda\mid_{\frakt_0}=0$, implying $\lambda={\chi}$	
\end{proof}

Fix a $\bbq$-simple system 
\[
\Delta_{\bbq}=\left\{ \alpha_{1},\dots,\alpha_{r}\right\} .
\]
Let $\chi_{1},\dots,\chi_{r}\in\fraks^{*}$ be the\textbf{ $\bbq$-fundamental
	weights}, i.e. for any $1\leq i,j\leq r$ 
\begin{equation}
\left\langle \chi_{i},\alpha_{j}\right\rangle =\delta_{i,j}\label{eq:fundamental weights}
\end{equation}
(Kronecker delta). 

For any $c_{1},\dots,c_{r}\in\bbz$ the character \[
\chi=\sum_{i=1}^{r}c_{i}\chi_{i}
\]
is called a \textbf{$\bbq$-abstract weight} and if  $c_{1},\dots,c_{r}$ are non-negative, then $\chi$ is called \textbf{dominant}.

A $\bbq$-weight $\chi$ for $\varrho$ is called a \textbf{$\bbq$-highest  weight for $\varrho$} if any $\lambda\in\Phi_{\varrho}$ satisfies $\lambda\leq\chi$.

\begin{defn}
	\label{def:strongly rational}A finite-dimensional $\bbq$-representation
	$\varrho:G\rightarrow\mbox{GL}\left(V\right)$ is called \textbf{strongly rational over $\bbq$} if there is a $\bbq$-highest weight for $\varrho$ and the $\bbq$-weight vector space for the $\bbq$-highest weight is of dimension one.
\end{defn}

\begin{lem}
	\label{lem: exist rho 1}\cite[\S 12]{key-27} Let $\chi$ be a $\bbq$-abstract
	weight. 
	\begin{enumerate}
		\item[$\left(i\right)$]  $\chi$ is conjugated under $W\left(\Phi_{\bbq}\right)$ to precisely one
		$\bbq$-abstract dominant weight. 
		\item[$\left(ii\right)$] If $\chi$ is a $\bbq$-abstract dominant weight, then there exist a positive integer $m$ and a strongly rational over $\bbq$-representation $\varrho:G\rightarrow\mbox{GL}\left(V\right)$ with $\bbq$-highest weight $m\chi$. 
	\end{enumerate}
\end{lem}
 
 Recall that $\Gamma=G(\bbz)$ and $T$ is a maximal $\bbr$-split torus in $G$. 

\begin{lem}	\label{lem:factor equiv}
	Let $A$ be a subgroup of $T$, $G=G_{1}\times G_{2}$ be an almost $\bbq$-direct product, and for $i=1,2$, $A_i:=(A\cap G_i)^\circ$, $\Gamma_i:=G_i(\bbz)$. Then, there exist divergent trajectories for the action of $A$ on $G/\Gamma$ if and only if one of the following is satisfied: 
	\begin{enumerate}
		\item 
		There exist divergent trajectories for the action of $A_i$ on $G_{i}/\Gamma_i$ for both $i=1,2$.  
		\item For some $\{i,j\}=\{1,2\}$, $\dim A_i=0$ and there exist divergent trajectories for the action of $A_j$ on $G_{j}/\Gamma_j$. 
	\end{enumerate}
\end{lem}

\begin{proof}
	Since $G_1\times G_2$ is an almost $\bbq$-direct product, $\Gamma$ is commensurable with $\Gamma_{1}\times \Gamma_2$. Thus, by \cite[Lemma 6.1(2)]{key-3} we may replace $\Gamma$ with $\Gamma_{1}\times\Gamma_{2}$.
	
	Assume $(1)$, i.e., for $i=1,2$ there exists $x_i=g_i\Gamma_{i}\in G_{i}/\Gamma_{i}$ so that $A_{i}x_i$ is a divergent trajectory. Since $G_1$ and $G_2$ commute and $A=A_1 A_2$, we get that $Ag_1g_2\Gamma$ is also a divergent trajectory.
	

	Assume $(2)$, i.e., $\dim(A_1)=0$ and there exists $x=g\Gamma_{2}\in G_{2}/\Gamma_{2}$ so that $A_{2}x$ diverges. Since $A=A_1A_2$, any unbounded sequence in $A$ has an unbounded projection to $A_2$. Thus, $Ag\Gamma$ diverges in $G/\Gamma$.
	
	Assume there exists a divergent trajectory $Ax$ for $x\in G/\Gamma$ and that $\dim A_1\neq 0$. Then, $A_1 x$ also diverges in $G/\Gamma$. Since $G_1$ commutes with $G_2$, there must be a divergent trajectory for $A_1$ in $G_1/\Gamma_1$.
\end{proof}

\begin{thm} \cite[Corollary 2]{w Q-rank} \label{thm:big dim}
	Let $A$ be a subgroup of $G$ such that $\dim A>\mathrm{rank}_\bbq G$. Then, there are no divergent trajectories for the action of $A$ on $G/\Gamma$.
\end{thm}

\begin{proof}[Proof of Theorem \ref{thm:2}]
	We prove the theorem via induction on the number of $\bbq$-factors of $G$.   
		
	First, assume that $G$ is almost $\bbq$-simple. In this case the assumptions of the theorem are equivalent to 
	\begin{equation}\label{eq:dim equiv}
	0<\dim(A)\leq\mathrm{rank}_\bbq G.
	\end{equation}
	
	Assume (\ref{eq:dim equiv}). Then according to Theorem \ref{thm:maximal almost split}, there exists $n\in N_G\left(T\right)$ such that $\mbox{Ad}\left(n\right)A$ is almost $\bbq$-split. We claim that $An^{-1}\Gamma$ is a divergent trajectory in $G/\Gamma$. 
	It is equivalent to showing that $\mbox{Ad}\left(n\right)A\Gamma$ diverges in $G/\Gamma$. 
	
	Let $\left\{a_k\right\}$ be an unbounded sequence in the Lie algebra of $\mbox{Ad}\left(n\right)A$. Since $\mbox{Ad}\left(n\right)A$ is almost $\bbq$-split and the $\bbq$-fundamental weights span $\fraks^*$, we may deduce that for some $1\leq i\leq r$ the sequence ${\chi}_i(a_{k})$ is unbounded. 
	
	Assume there exists a subsequence $\left\{a_{k_\ell}\right\}$ such that 
	\begin{equation}\label{eq:chi +infty}
	{\chi}_i(a_{k_\ell})\rightarrow+\infty. 
	\end{equation}
	By Lemma \ref{lem: exist rho 1} there exist a positive integer $m$ and a strongly rational over $\bbq$-representation $\varrho$ with a $\bbq$-highest weight $m\cdot w\left(-\chi_i\right)$, for some $w\in W\left(\Phi_{\bbq}\right)$. The set of $\bbq$-weight of $\varrho$ are invariant under the action of the $\bbq$-Weyl group, and so is their dimension (see \cite[\S 12]{key-27}). Hence, Definition \ref{def:strongly rational} implies that the subspace $V_{\varrho,-m\chi_i}$ is one dimensional.  
	Let $v\in V_{\varrho,m\chi_i}$ be a rational vector. Then, Lemma \ref{lem:chi tilde} and (\ref{eq:chi +infty}) imply 
	\[
	\varrho\left(\exp\left(a_{k_\ell}\right)\right)v=e^{-m\tilde{\chi}_{i}\left(a_{k_\ell}\right)}v\rightarrow 0.
	\]
	
	The case \[
	{\chi}_i(a_{k_\ell})\rightarrow-\infty,
	\]
	can be shown in a similar way.
	
	If (\ref{eq:dim equiv}) is not satisfied, then either $\dim A=0$, in which case, there are obviously no divergent trajectories for the action of $A$ on $G/\Gamma$, or $\dim A>\rank_{\bbq}G$, and then Theorem \ref{thm:big dim} implies the same conclusion. Proving the claim for the case $G$ is almost $\bbq$-simple. 
	 
	Second,	assume that $G=G_1\times G_2$ is a non-trivial $\bbq$-almost direct product. Let $A_1:=(A\cap G_1)^\circ$ and $A_2:=(A\cap G_2)^\circ$. 
	Then, (\ref{eq:main}) is equivalent to the following:
	\begin{equation} \label{eq:A1,A2 def}
	A=A_1 A_2,\quad \dim A_1\le\rank_\bbq G_1,\quad \mathrm{and}\quad \dim A_2\le\rank_\bbq G_2.
	\end{equation}
	
	The claim now follows from the induction assumption and Lemma \ref{lem:factor equiv}.
\end{proof}

\section{A special case of Theorem \ref{thm:main}}\label{sec: example}

Before presenting the proof of Theorem \ref{thm:main}, we show an unknown case of it (see \cite[Corollary 4.15]{w}). It allows us to explain the main ideas of the general case.  

\begin{example}\label{ex: weiss}
	Assume $G = \sl_4(\bbr)$, $\Gamma=\sl_4(\bbz)$, and \[	A=\left\{\begin{pmatrix}e^s & & & \\		& e^{-s} & & \\		& & e^t & \\	& & & e^{-t}		\end{pmatrix}\right\}.\]
	Then, there exist non-obvious divergent trajectories for the action of $A$ on $G/\Gamma$.  
\end{example}

\begin{proof}
	Note that $A$ is the subgroup of \[T:=\left\{\begin{pmatrix}	e^{t_1} & & & \\	& e^{t_2} & & \\	& & e^{t_3} & \\	& & & e^{t_4}	\end{pmatrix}\::\:\sum_{i=1}^4 t_i=0,\right\}\] which is defined by $\chi=0$ for $\chi:$\[	\chi(t_1,t_2,t_3,t_4):= t_1+t_2. \]  
	
	Recall that a $\bbq$-simple system here is \[	\alpha_i:=t_i-t_{i+1},\quad i=1,2,3.\]	Hence, 	
	\begin{equation}\label{eq: sl4 character}
	\chi=\frac{1}{3}(\alpha_1+2\alpha_2+\alpha_{3}). 
	\end{equation}	
	
	Let $\{a_k\}$ be an unbounded sequence in $A$. Then, there exists $1\le i\le4$ such that the sequence $\{\alpha_{i}(a_k)\}$ is unbounded. 	
	Since all elements in $A$ satisfy (\ref{eq: sl4 character}),  there exists a sub-sequence $\{a'_k\}\subset\{a_k\}$ such that for some $1\leq \ell\neq j\leq 3$ \begin{equation}\label{eq: unbounded seq}
	\alpha_{\ell}(a'_k)\rightarrow+\infty,\quad \alpha_{j}(a'_k)\rightarrow-\infty.
	\end{equation} 
	Equation (\ref{eq: unbounded seq}) can also be verified using that for any element $a\in A$, \[
	\alpha_{1}(a)=2s(a),\quad\alpha_{2}(a)=(-s-t)(a),\quad\alpha_{3}(a)=2t(a). \]
	
	Let $\varrho_{1}:G\rightarrow\mathrm{GL}(\bbr^4)$ be the standard representation. For $i=2,3$ let $\varrho_{i}:G\rightarrow\mathrm{GL}(\bigwedge^i\bbr^4)$ be the $i$-th wedge product of the standard representation. Let $e_1,e_2,e_3,e_4$ be the standard basis of $\bbr^4$. Let
	\begin{eqnarray*}
		&&v_1^+:=e_1,\quad v_2^+:=e_2\wedge e_4,\quad v_3^+:=e_1\wedge e_2\wedge e_3,\\
		&&v_1^-:=e_4,\quad v_2^-:=e_1\wedge e_3,\quad v_3^-:=e_2\wedge e_3\wedge e_4,
	\end{eqnarray*}
	and \[
	P_+:=\left\{\begin{pmatrix}	\star & 0 & \star & 0 \\	
	0 & \star & \star & \star \\	
	0 & 0 & \star & 0 \\	
	0 & 0 & 0 & \star	\end{pmatrix}\right\},\quad P_-:=\left\{\begin{pmatrix}	\star & 0 & 0 & 0 \\	
	0 & \star & 0 & 0 \\	
	\star & \star & \star & 0 \\	
	0 & \star & 0 & \star	\end{pmatrix}\right\}. \]
	Then,  $\varrho_{i}^{\pm}:=\varrho_{i}$, $v^\pm_i$, $P_\pm$ satisfy $(ii)$, $(iii)$, $(iv)$, $(v)$, and $(vi)$.
	
	For any $a\in A$ we have \[
	\varrho_2(a)v^{\pm}_i=(-\chi\pm\alpha_{i})(a)v^{\pm}_i\]
	and for $i=1,3$ \[
	\varrho_i(a)v^{\pm}_i=\frac{1}{2}(\chi\pm\alpha_{i})(a)v^{\pm}_i. \]
	Thus, $(i)$ follows from (\ref{eq: unbounded seq}). 	
\end{proof}

\section{\label{sec:Rational Characters and Divergent Trajectories}Rational Characters and Divergent Trajectories}
We preserve the notation of \S 3 and \S 4. 

The goal of this section is to prove Theorem \ref{thm:main}. For the proof we need the following results:

\begin{lem}
	\label{lem: chi_i all pos}For any $i=1,\dots,r$ there exists $d_{i}>0$ such that 
	\[
	\chi_{i}=d_{i}\sum_{\beta\in\Phi_{\bbq},\beta\geq\alpha_{i}}\beta
	\]
	(where $\chi_1,\dots,\chi_r$ are defined in (\ref{eq:fundamental weights})). 
\end{lem}
\begin{proof}
	Let $1\leq i\neq j\leq r$ and denote $\gamma_{i}=\sum_{\beta\in\Phi_{\bbq},\beta\geq\alpha_{i}}\beta$.
	By (\ref{eq: Weyl group def}) for any $\beta\in\Phi_{\bbq}$,  $\beta\geq\alpha_{i}$ implies 	$s_{\alpha_{j}}\left(\beta\right)\geq\alpha_{i}$.
	Hence, $w_{\alpha_{j}}$ is bijective on the set 
	\[
	\left\{ \lambda\in\Phi_{\bbq}\::\:\lambda\geq\alpha_{i}\right\} .
	\]
	Therefore, $w_{\alpha_{j}}\left(\gamma_{i}\right)=\gamma_{i}$.
	Using (\ref{eq: Weyl group def}) we may deduce $\left\langle \gamma_{i},\alpha_{j}\right\rangle =0$. 
	
	Since $\gamma_{i}$ is a sum of positive roots, there exist non-negative integers $k_{1},\dots,k_{r}$ such that $\gamma_{i}=\sum_{j=1}^{r}k_{j}\alpha_{j}$.
	Since the Killing form is nondegenerate on $\mathfrak{s}^{*}$ and $\gamma_{i}\neq0$, we arrive at 
	\[
	0<\left\langle \gamma_{i},\gamma_{i}\right\rangle =\sum_{j=1}^{r}k_{j}\left\langle \gamma_{i},\alpha_{j}\right\rangle =k_{i}\left\langle \gamma_{i},\alpha_{i}\right\rangle .
	\]
	Thus $\left\langle \gamma_{i},\alpha_{i}\right\rangle >0$ and so (\ref{eq:fundamental weights}) implies that there exists $d_{i}>0$ which satisfies the conclusion of the lemma. 
\end{proof}

\begin{lem}
	\label{lem:w(chi)+lambda}Let $\varrho:G\rightarrow\mbox{GL}\left(V\right)$ be a strongly rational over $\bbq$-representation with a $\bbq$-highest weight $\chi$. If $w\in W\left(\Phi_\bbq\right)$ and $\beta\in\Phi_\bbq$ satisfy $w\left(\chi\right)+\beta\in\Phi_{\varrho}$, then $\left\langle w\left(\chi\right),\beta\right\rangle <0$ (see (\ref{eq: Weyl group def}) for the definition of the Weyl group). 
\end{lem}
\begin{proof}
	Assume $w\in W\left(\Phi_\bbq\right)$ and $\beta\in\Phi_\bbq$ satisfy $w\left(\chi\right)+\beta\in\Phi_{\varrho}$. Since $\Phi_{\bbq}$ is invariant under the action of the Weyl group, we have  $\chi+w^{-1}\left(\beta\right)\in\Phi_{\varrho}$.  
	Since $\chi$ is a $\bbq$-highest weight, we may deduce 	
	\begin{equation}\label{eq:w_beta<0}
	w^{-1}\left(\beta\right)<0.
	\end{equation}
	
	Let $l:=\left\langle w\left(\chi\right),\beta\right\rangle$. Since the Killing form is invariant under  automorphisms of $\frakg$ we have  $l=\left\langle \chi,w^{-1}\left(\beta\right)\right\rangle $. Then, by (\ref{eq: Weyl group def}) and the linearity of the killing form we have 
	\begin{equation}\label{eq:l}
	s_\beta\left(\chi+w^{-1}\left(\beta\right)\right)=\chi-\left(l+1\right)w^{-1}\left(\beta\right).
	\end{equation}
	Since $\chi$ is a $\bbq$-highest weight, it follows from (\ref{eq:w_beta<0}) and (\ref{eq:l}) that $l\leq -1$.	 
\end{proof}

The following lemma can be proved in a similar way to Lemma \ref{lem:factor equiv}. 
\begin{lem}	\label{lem:factor for main}
	Let $A$ be a subgroup of $T$ and $G=G_{1}\times G_{2}$ be an almost $\bbq$-direct product, and for $i=1,2$, $A_i:=(A\cap G_i)^\circ$, $\Gamma_{i}:=G_i(\bbz)$. If there exist divergent trajectories for the action of $A_1$ on $G_1/\Gamma_1$ and non-obvious divergent trajectories for the action of $A_2$ on $G_2/\Gamma_2$, then there exist non-obvious divergent trajectories for the action of $A$ on $G/\Gamma$. 
\end{lem}

\begin{lem}\label{lem:highest root}
	\cite[\S VI 1.8]{bou} If $\Phi_\bbq$ is irreducible, then there exists a root $\beta\in\Phi_\bbq$ such that every  $\alpha\in\Phi_\bbq$ satisfies $\beta\geq\alpha$.
\end{lem}

\begin{lem}	\label{lem:dim weyl}\cite[§12]{key-27}	Let $\varrho:G\rightarrow\mbox{GL}\left(V\right)$ be a strongly rational over $\bbq$ representation with a $\bbq$-highest weight $\chi$ and $w\in W\left(\Phi_\bbq\right)$. Then the dimension of $V_{\varrho,w\left(\chi\right)}$ is one, in particular, $w\left(\chi\right)$ is a $\bbq$-weight for $\varrho$.\end{lem}

\begin{thm}\label{thm: weiss for epsilon case}
	\cite[Thm. 4.13]{w} Let $G$
	be a semisimple $\bbq$-algebraic group, $\Gamma=G\left(\bbz\right)$,
	and $A$ be a subgroup of $T$. Suppose that there are subgroups $P_{+},P_{-}$,
	finitely many $\bbq$-representations $\varrho_{i}^{+}:G\rightarrow\mbox{GL}\left(V_{i}^{+}\right)$,
	$\varrho_{i}^{-}:G\rightarrow\mbox{GL}\left(V_{i}^{-}\right)$, non-zero
	vectors $v_{i}^{+}\in V_{i}^{+}\left(\bbq\right)$, $v_{i}^{-}\in V_{i}^{-}\left(\bbq\right)$,
	such that the following hold:
	\begin{enumerate}
		\item[$\left(i\right)$] For any unbounded sequence $\left\{ a_{k}\right\} \subset A$ there
		is a subsequence $\left\{ a_{k}^{\prime}\right\} $ and $i$ such
		that $\varrho_{i}^{\pm}\left(a_{k}^{\prime}\right)v_{i}^{\pm}\underset{n\rightarrow\infty}{\longrightarrow}0$. 
		\item[$\left(ii\right)$] For each $i$, $\varrho_{i}^{\pm}\left(P_{\pm}\right)$ leaves the
		line $\bbr\cdot v_{i}$ invariant. 
		\item[$\left(iii\right)$] $P_{\pm}=\overline{P_{\pm}\cap G\left(\bbq\right)}$. 
		\item[$\left(iv\right)$] $T\subset P_{\pm}$.
		\item[$\left(v\right)$] For any $\bbr$-root $\alpha$, if $\frakg_{\alpha}\cap\mbox{Lie}\left(P_{\pm}\right)\neq\left\{ 0\right\} $
		then $\frakg_{\alpha}\subset\mbox{Lie}\left(P_{\pm}\right)$. 
		\item[$\left(vi\right)$] $P_{+}$ and $P_{-}$ generate $G$.
	\end{enumerate}
	Then there are non-obvious divergent trajectories for $A$.
\end{thm}

\begin{proof}[{Proof of Theorem \ref{thm:main}}]
	
Let $A$ be a subgroup of $T$ which satisfy the assumptions. Note that by Theorem \ref{thm:2} and Lemma \ref{lem:factor for main} we may assume that $G$ is almost $\bbq$-simple. In particular, $\Phi_{\bbq}$ is irreducible and $0<\dim A<\mbox{rank}_{\bbq}G$.

We show the claim by finding representations, vectors, and subgroups which satisfy the assumptions of Theorem \ref{thm: weiss for epsilon case} for $A$.

According to Theorem \ref{thm:maximal almost split}, applying conjugation we may assume that $A$ is almost $\bbq$-split.
Since $\dim A<\mbox{rank}_{\bbq}G$, Lemma \ref{lem:exists Q-caracter} implies that there exists a non-zero $\bbq$-character $\chi$ which is trivial on $A$. 

First, for some choice of a $\bbq$-simple system, $\chi$ is in the positive $\bbq$-Weyl chamber. That is, there exist $b_{1},\dots,b_{r}\geq0$, not all zero, such that 
\begin{equation}\label{eq:chi decomp into chi_i}
\chi=\sum_{i=1}^{r}b_{i}\chi_{i}.
\end{equation}
Since $\Phi_\bbq$ is irreducible, it follows from Lemma \ref{lem: chi_i all pos} that there exist positive $d_{1},\dots,d_{r}$ such that 
\begin{equation}\label{eq: chi decomp into alpha_i}
\chi=\sum_{i=1}^{r}d_{i}\alpha_{i}.
\end{equation}
Assume that  $\left\{ a_{k}\right\}$ is an unbounded sequence in $A$.
Then, for some $1\le i\le r$, the sequence $\alpha_{i}(a_k)$ is unbounded. It follows from (\ref{eq: chi decomp into alpha_i}) that after passing to a subsequence $\left\{ a_{k}^{\prime}\right\} \subset\left\{ a_{k}\right\} $,
there exist $1\leq\ell,j\leq r$ such that 
\begin{equation}
\alpha_{\ell}\left(a_{k}^{\prime}\right)\rightarrow+\infty\quad\mbox{ and }\quad\alpha_{j}\left(a_{k}^{\prime}\right)\rightarrow -\infty.\label{eq:a_i, a_j unbounded}
\end{equation}

We now explain the main idea of the rest of the proof. 
If $b_1,\dots,b_r$ in (\ref{eq:chi decomp into chi_i}) are all positive integers, then, up to replacing $\chi$ with an integer multiple of it, $\chi+\alpha_{i}$, for each $1\le i\le r$, is a $\bbq$-abstract dominant weight. Therefore, for any $1\le i\le r$, one can construct $\bbq$-representations $\varrho_{i}^{+}$ and $\varrho_{i}^{-}$ with $\bbq$-highest weight $\chi+\alpha_{i}$ and $\bbq$-lowest weight $-\chi-\alpha_{i}$, respectively (see Lemma \ref{lem: exist rho 1}). Then, in a similar way to the proof of Example \ref{ex: weiss}, for any $1\le i\le r$, we could take $v_{i}^{+}$ and $v_{i}^{-}$ to be the $\bbq$-highest weight vector of $\varrho_{i}^{+}$ and $\bbq$-lowest weight vector of $\varrho_{i}^{-}$, respectively. In that case, as $P_{+}$ and $P_{-}$ one could take two opposite minimal $\bbq$-parabolic groups. 

Since  $b_1,\dots,b_r$ in (\ref{eq:chi decomp into chi_i}) are not necessarily integers, we start by finding a rational approximation $\chi'$ to $\chi$. 
Note that the coefficients $b_1,\dots,b_r$ in (\ref{eq:chi decomp into chi_i}) are also not necessarily all positive. Lemma \ref{lem: exist rho 1} can still be used to construct representations with $\chi+\alpha_i$, for $1\leq i\leq r$, as $\bbq$-weights, but they are not necessarily the $\bbq$-highest weights of these representations. In particular, we may not be able to choose $P_{\pm}$ to be $\bbq$-parabolic groups. We can however assume that one of $b_1,\dots,b_r$ is positive. It allows us to construct $P_{\pm}$ which satisfy the assumptions of Theorem \ref{thm: weiss for epsilon case}, but it requires some subtleties.  


Let 
\begin{equation}
R=\max\left\{ \sum_{j=1}^{r}\left\langle \chi_{i},\chi_{j}\right\rangle \::\:1\leq i\leq r\right\} .\label{eq:R def}
\end{equation}
Note that the value of $R$ does not depend on the choice of a simple system.

Let $q,p_{1},\dots,p_{r}\in\bbz$, $q\neq0$, satisfy to following for all $i=1,\dots,r$ 
\begin{equation}
\left|qb_{i}-p_{i}\right|<\frac{1}{2Rr}.\label{eq:(bi-pi)}
\end{equation}
The existence of such an approximation follows from Dirichlet's theorem on simultaneous approximation, though the latter is stronger. 
Let 
\begin{equation}
\chi^\prime=\sum_{i=1}^{r}p_{i}\chi_{i}.\label{eq:chi_i decom}
\end{equation}
Then, by taking a better approximation (possibly a larger $q$), we can assume that 
\begin{equation}
p_{1},\dots,p_{r}\ge0.\label{eq:a_i non-neg}
\end{equation}

Since we can assume $\chi'$ is not trivial (again, by possibly taking a smaller $R$), it follows from (\ref{eq:fundamental weights}), (\ref{eq:chi_i decom}), and (\ref{eq:a_i non-neg}) that there exists $1\le i\le r$ such that 
\begin{equation}
\left\langle\chi^\prime,\alpha_{i}\right\rangle =p_{i}>0.\label{eq:inner ai}
\end{equation}
Thus, by replacing $\chi$ with a positive integer multiplication of it, we may also assume that for $i$ and any $l=1,\dots,r$
\begin{equation}
\left\langle\chi^\prime,\alpha_{i}\right\rangle >\left\langle \alpha_{l},\beta\right\rangle .\label{eq:chi,alpha big}
\end{equation}

Let $1\le i\le r$. The character $\chi^\prime-\alpha_{i}$ is a $\bbq$-abstract weight. According to Lemma \ref{lem: exist rho 1}$\left(i\right)$ there exists $w_{i}\in W\left(\Phi_\bbq\right)$ such that $w_{i}\left(\chi^\prime-\alpha_{i}\right)$ is a $\bbq$-abstract dominant weight. Then, according to Lemma \ref{lem: exist rho 1}$\left(ii\right)$, there exists a strongly rational over $\bbq$ representation $\varrho_{i}^{+}:G\rightarrow\mbox{GL}\left(V_{i}^{+}\right)$ (respectively $\varrho_{i}^{-}:G\rightarrow\mbox{GL}\left(V_{i}^{-}\right)$) with a $\bbq$-highest weight $m_{i}^{+}\cdot w_{i}\left(\chi^\prime-\alpha_{i}\right)$ (respectively $\bbq$-lowest weight $m_{i}^{-}\cdot w_{i}\left(-\chi^\prime+\alpha_{i}\right)$) for some $m_{i}^{+},m_{i}^{-}\in\bbn$. 
Let $v_{i}^{\pm}$ be a non-zero weight vector for $\pm\left(\chi^\prime-\alpha_{i}\right)$ in $V_{i}^{\pm}\left(\bbq\right)$ (such a vector exists by Lemma \ref{lem:dim weyl}). 

Let \[
\Psi=\left\{\beta\in\Phi_\bbq\::\:\beta\geq\alpha_{i}\right\}.\]
Then, $\Psi$ is a closed subset of $\Phi_\bbq$. Therefore, by \cite[Prop. 21.9(ii)]{Bor} there are unique closed connected unipotent $\bbq$-subgroups $U_{+}$, $U_{-}$ normalized by $Z\left(S\right)$ (the centralizer of $S$) with Lie algebras  $\bigoplus_{\beta\in\Psi}\frakg_{\beta}$, $\bigoplus_{\beta\in\Psi}\frakg_{-\beta}$, respectively. Let $P_{\pm}=T\cup U_{\pm}$. 

We now want to show that $\varrho_{i}^{\pm}$, $v_{i}^{\pm}$, and $P_{\pm}$ satisfy the assumptions of Theorem \ref{thm: weiss for epsilon case} for $A$. We start by showing $(i)$. 

Let $\left\{ a_{k}\right\} $ be an unbounded sequence in $A$.
After passing to a subsequence $\left\{ a_{k}^{\prime}\right\} \subset\left\{ a_{k}\right\}$, there exist $1\leq\ell,j\leq r$ which satisfy (\ref{eq:a_i, a_j unbounded}). Hence, up to passing to a subsequence, we may assume 
\begin{equation}
\alpha_{\ell}\left(a_{k}^{\prime}\right)=\max_{1\leq i\leq r}\alpha_{i}\left(a_{k}^{\prime}\right)>0\quad\mbox{ and }\quad\alpha_{j}\left(a_{k}^{\prime}\right)=\min_{1\leq i\leq r}\alpha_{i}\left(a_{k}^{\prime}\right)<0.\label{eq:a_i, a_j def}
\end{equation}

By (\ref{eq:fundamental weights}), (\ref{eq:R def}), and (\ref{eq:a_i, a_j def})
for $i=1,\dots,r$ 
\begin{equation}
\chi_{i}\left(a_{k}^{\prime}\right)=\sum_{m=1}^{r}\left\langle \chi_{i},\chi_{m}\right\rangle \alpha_{m}\left(a_{k}^{\prime}\right)\leq R\alpha_{\ell}\left(a_{k}^{\prime}\right),\label{eq:chi_i upper}
\end{equation}
and 
\begin{equation}
\chi_{i}\left(a_{k}^{\prime}\right)=\sum_{m=1}^{r}\left\langle \chi_{i},\chi_{m}\right\rangle \alpha_{m}\left(a_{k}^{\prime}\right)\geq R\alpha_{j}\left(a_{k}^{\prime}\right).\label{eq:chi_i lower}
\end{equation}
Then, it follows from the choice of $\varrho_{\ell}^+$, $v_\ell^+$,  (\ref{eq:chi decomp into chi_i}), (\ref{eq:(bi-pi)}), (\ref{eq:chi_i decom}), and (\ref{eq:chi_i upper})
that
\begin{eqnarray*}
\varrho_{\ell}^{+}\left(a_{k}^{\prime}\right)v_{\ell}^{+} & = & \exp\left[\left(\chi^\prime-\alpha_{\ell}\right)\left(a_{k}^{\prime}\right)\right]\cdot v_{\ell}^{+}\\
 & = & \exp\left[\left(\chi+\sum_{i=1}^{r}\left(p_{i}-b_{i}\right)\chi_{i}-\alpha_{\ell}\right)\left(a_{k}^{\prime}\right)\right]\cdot v_{\ell}^{+}\\
 & \leq & \exp\left[\left(\frac{1}{2Rr}Rr\alpha_{\ell}-\alpha_{\ell}\right)\left(a_{k}^{\prime}\right)\right]\cdot v_{\ell}^{+}\\
 & = & \exp\left[-\frac{1}{2}\alpha_{\ell}\left(a_{k}^{\prime}\right)\right]\cdot v_{\ell}^{+}\underset{k\rightarrow\infty}{\longrightarrow}0.
\end{eqnarray*}
In a similar way, it follows from the choice of $\varrho_{j}^+$, $v_j^+$, (\ref{eq:chi decomp into chi_i}), (\ref{eq:(bi-pi)}), (\ref{eq:chi_i decom}), and (\ref{eq:chi_i lower}) that 
\[
\varrho_{j}^{-}\left(a_{k}^{\prime}\right)v_{j}^{-}\underset{k\rightarrow\infty}{\longrightarrow}0.
\]
Therefore, $\left(i\right)$ is satisfied. 

Next, we show $(ii)$. Let $\beta\in\Psi$ and $1\leq l\leq r$. 
If $\left\langle\alpha_{l},\beta\right\rangle\leq 0$, then \[
\left\langle \chi^{\prime}-\alpha_{l},\beta\right\rangle \geq\left\langle \chi^{\prime},\beta\right\rangle \geq 0.
\]
If $\left\langle\alpha_{l},\beta\right\rangle> 0$, then there exists $1\leq j\leq q$ such that $\alpha_{l},\beta\in\Phi_j$. Since in that case $\beta\geq\alpha_{i_j}$, using (\ref{eq:a_i non-neg}), (\ref{eq:inner ai}), and (\ref{eq:chi,alpha big}) we arrive at
\begin{eqnarray*}
\left\langle \chi^{\prime}-\alpha_{l},\beta\right\rangle  & = & \left\langle \chi^{\prime},\beta\right\rangle -\left\langle \alpha_{l},\beta\right\rangle \\
 & \geq & \left\langle \chi^{\prime},\alpha_{i_j}\right\rangle -\left\langle \alpha_{l},\beta\right\rangle \\
 & > & 0.
\end{eqnarray*}
In view of Lemma \ref{lem:w(chi)+lambda}, we can deduce that in both cases 
\begin{equation}
\pm\left(\chi^{\prime}-\alpha_{l}\right)\pm\beta\notin\Phi_{\varrho_{l}^{\pm}}.\label{eq:lambda stab}
\end{equation}
It then follows from (\ref{eq: lambda acts on chi}) and the definition of $U^{\pm}$ that for $i=1,\dots,r$ the subspace  $V_{\varrho_{i}^{\pm},\pm\left(\chi^{\prime}-\alpha_{i}\right)}$ is invariant under $\varrho_{i}^{\pm}\left(U^{\pm}\right)$. 
Since any $\bbq$-weight space is invariant under $\varrho_{i}^{\pm}(T)$, $V_{\varrho_{i}^{\pm},\pm\left(\chi^{\prime}-\alpha_{i}\right)}$ is invariant under  $\varrho_{i}^{\pm}(P_\pm)$.
By Lemma \ref{lem:dim weyl}, $V_{\varrho_{i}^{\pm},\pm\left(\chi^{\prime}-\alpha_{i}\right)}$ is one dimensional. That is, $\bbr\cdot v_{i}^{\pm}=V_{\varrho_{i}^{\pm},\pm\left(\chi^{\prime}-\alpha_{i}\right)}$. Proving $\left(ii\right)$.

Last, it is easy to see that $\left(iii\right),\left(iv\right),\left(v\right)$ are satisfied, and since $U^{+},U^{-}$ are unipotent radicals of opposite parabolic subgroups of $G$ which are not contained in any proper $\bbq$-factor of $G$, \cite[Prop. 4.11]{key-27} implies
$\left(vi\right)$.%

\end{proof}

\end{document}